\documentclass[12pt,reqno]{amsart}
\usepackage{amsmath,amssymb,amsthm,mathtools,calc,verbatim,enumitem,tikz,url,hyperref,mathrsfs,cite,fullpage}
\usepackage{bbm}
\usepackage{stmaryrd}
\usepackage{textcomp}
\usepackage{setspace}
\usepackage{amssymb}
\usepackage{amsthm}
\usepackage{amsmath}
\usepackage{graphicx}
\usepackage{marvosym}
\usepackage{empheq}
\usepackage{latexsym}
\usepackage{fontenc}
\usepackage{color}
\usepackage{cleveref}

\newtheorem{theorem}{Theorem}[section]
\newtheorem{lemma}[theorem]{Lemma}

\newtheorem*{claim*}{Claim}

\theoremstyle{definition}

\newtheorem*{qu*}{Question}
\theoremstyle{remark}

\newcommand\N{\mathbb{N}}
\newcommand\R{\mathbb{R}}
\newcommand\Z{\mathbb{Z}}
\newcommand\Ex{\mathbb{E}}
\newcommand\E{\operatorname{\mathbb{E}}}

\newcommand\U{\mathcal{U}}
\newcommand\cU{\mathcal{U}}

\renewcommand\Pr{\operatorname{\mathbb{P}}}

\renewcommand\leq{\leqslant}
\renewcommand\geq{\geqslant}
\renewcommand\le{\leqslant}
\renewcommand\ge{\geqslant}
\renewcommand\to{\rightarrow}

	\def\R{\mathbb{R}}
	\def\Z{\mathbb{Z}}

	\def\N{\mathbb{N}}

	\def\1{\mathbbm{1}}

	\def\<{\langle }
	\def\>{\rangle }

	\DeclareMathOperator{\inte}{int}
	\DeclareMathOperator{\unif}{Unif}
	\def\vol{\mathrm{Vol}}

\begin{document}

\title{Towards Hadwiger's conjecture via Bourgain Slicing}
\author{Marcelo Campos \and Peter van Hintum \and Robert Morris \and Marius Tiba}

\address{IMPA, Estrada Dona Castorina 110, Jardim Bot\^anico,
Rio de Janeiro, 22460-320, Brazil}\email{marcelo.campos@impa.br}

\address{Mathematical Institute, University of Oxford, Radcliffe Observatory Quarter, Woodstock Road, Oxford, OX2 6GG, UK}\email{peter.vanhintum@new.ox.ac.uk}

\address{IMPA, Estrada Dona Castorina 110, Jardim Bot\^anico,
Rio de Janeiro, 22460-320, Brazil}\email{rob@impa.br}

\address{IMPA, Estrada Dona Castorina 110, Jardim Bot\^anico,
Rio de Janeiro, 22460-320, Brazil}\email{mt576@cam.ac.uk}

\thanks{The first author was partially supported by CNPq, and the third author was partially supported by FAPERJ (Proc.~E-26/202.993/2017) and CNPq (Proc.~303681/2020-9)}

\begin{abstract}
In 1957, Hadwiger conjectured that every convex body in $\R^d$ can be covered by $2^d$ translates of its interior. For over 60 years, the best known bound was of the form $O(4^d \sqrt{d} \log d)$, but this was recently improved by a factor of $e^{\Omega(\sqrt{d})}$ by Huang, Slomka, Tkocz and Vritsiou. In this note we take another step towards Hadwiger's conjecture by deducing an almost-exponential improvement from the recent breakthrough work of Chen, Klartag and Lehec on Bourgain's slicing problem. More precisely, we prove that, for any convex body $K \subset \R^d$, 
$$\exp\bigg( - \Omega\bigg( \frac{d}{(\log d)^8} \bigg) \bigg) \cdot 4^d$$
translates of $\inte(K)$ suffice to cover $K$. We also show that a positive answer to Bourgain's slicing problem would imply an exponential improvement for Hadwiger's conjecture.
\end{abstract}
	
\maketitle 

\section{Introduction}

Hadwiger's covering problem asks: how many translates of the interior of a convex body $K \subset \R^d$ are needed to cover $K$? That is, it asks for the value of
$$N(K) = \min\Big\{ N \in \N \,:\, \exists \, x_1,\ldots, x_N \in \R^d \,\text{ such that }\, K \subset \, \bigcup_{i=1}^N \big( x_i + \inte(K) \big) \Big\}.$$  
Hadwiger~\cite{H57} conjectured in 1957 that $N(K) \leq 2^d$ for all convex $K \subset \R^d$. Note that this bound is attained by the cube $[0,1]^d$. The conjecture was proved when $d \le 2$ by Levy~\cite{L55} in 1955, but for over 60 years the best known bound for general $d$ was 
$$N(K) \, \le \, \big( d\log d+ d\log \log d+ 5d \big) \binom{2d}{d} \, = \, O\big( 4^d \sqrt{d} \log d \big),$$ 
which follows from the Rogers--Shephard inequality~\cite{RS}, together with a bound of Rogers~\cite{R57} on the minimum density of a covering of $\R^d$ with translates of $K$. A few years ago, however, a breakthrough was made by Huang, Slomka, Tkocz and Vritsiou~\cite{HSTV}, who used a large deviation result of Guédon and Milman \cite{GM}, which is related to the so-called `thin-shell' phenomenon (see below), to obtain a bound of the form  
\begin{equation}\label{eq:HSTV}
N(K) \, \le \, e^{-\Omega(\sqrt{d})} \cdot 4^d.
\end{equation}
Here we will prove the following almost-exponential improvement of their bound. 

\begin{theorem}\label{cor:hadwiger}
If $K \subset \R^d$ is a convex body, then 
\[
N(K) \leq \exp\bigg( - \Omega\bigg( \frac{d}{(\log d)^8} \bigg) \bigg) \cdot 4^d
\]
as $d \to \infty$. 
\end{theorem}

We will deduce Theorem~\ref{cor:hadwiger} from recent results of Chen~\cite{C21} and Klartag and Lehec \cite{KL} on the Bourgain slicing problem, which asks for the smallest number $L_d > 0$ such that for every convex body $K \subset \R^d$ of volume $1$, there exists a hyperplane $H$ such that $K \cap H$ has $(d-1)$-dimensional volume at least $1 / L_d$. In particular, Bourgain~\cite{B86,B87} asked whether or not $L_d$ is bounded from above by an absolute constant. 
This problem is still open, and for many years the best known bound was of the form $L_d = O(d^{1/4})$, proved by Bourgain~\cite{B91,B02} (with an extra $\log$-factor) and Klartag~\cite{K06}. However, just a couple of years ago, Chen~\cite{C21} made a major breakthrough on the problem by proving a bound of the form $L_d = d^{o(1)}$. His bound was then improved further by Klartag and Lehec~\cite{KL}, who showed that $L_d = O(\log d)^4$. 

The breakthroughs in~\cite{C21} and~\cite{KL} both used ``stochastic localization", a powerful and beautiful technique that was introduced about ten years ago by Eldan~\cite{E13}, to bound the \emph{thin-shell constant}, $\sigma_d$, which is defined to be 
$$\sigma_d := \sup_{K} \, \Ex\big[ \big( \| X \|_2 - \sqrt{d} \big)^2 \big],$$
where the supremum is over convex bodies $K \subset \R^d$ in isotropic position\footnote{This means that $\Ex[X] = 0$ and $\Ex[ X \otimes X] = I_d$, where $X \sim \U(K)$ is a uniformly-chosen random point of $K \subset \R^d$, and $I_d$ is the identity matrix. For any convex body $K$ there exists a unique (up to rotations) affine transformation that maps $K$ to isotropic position.}, and $X \sim \cU(K)$ is a uniformly chosen random point of $K$. 
The thin-shell conjecture~\cite{ABP,BK} states that $\sigma_d = O(1)$, and it was shown by Eldan and Klartag~\cite{EK} that  
$$L_d = O(\sigma_d),$$
so bounds on the thin-shell constant imply bounds for the Bourgain slicing problem. We remark that, by a deep result of Eldan~\cite{E13}, bounds on the thin-shell constant also imply bounds for the Kannan--Lovász--Simonovits isoperimetric conjecture~\cite{KLS}, see e.g.~\cite{E13,LV}. 


We will use an equivalent formulation of the Bourgain slicing problem, due to Milman and Pajor~\cite{MP} (see also~\cite{KM}). Given a convex body $K \subset \R^d$, define the \emph{isotropic constant} of $K$ to be
$$L_K = \left(\frac{\sqrt{\det(\Sigma_K)}}{\vol_d (K)}\right)^{1/d},$$ 
where $\Sigma_K = \E[X \otimes X]$ is the covariance matrix of the random variable $X \sim \unif (K)$, that is, $X$ is a uniformly random point of $K$. Equivalently, there exists an affine transformation that maps $K$ to a convex body $K'$ of volume $1$ with $\Sigma_{K'} = L_K^2 I_d$, where $I_d$ is the identity matrix. By~\cite[Corollary~3.2]{MP} we have $L_K = \Theta(L_d)$ for every convex body $K \subset \R^d$, and hence $L_K = O(\log d)^4$, by the bound on $L_d$ proved by Klartag and Lehec~\cite{KL}. 

Our main result is the following bound on the covering number of a convex body. Since $L_K = O(\log d)^4$, it implies the bound in Theorem~\ref{cor:hadwiger} for Hadwiger's conjecture.

\begin{theorem}\label{main_thm}
If\/ $K \subset \R^d$ is a convex body, then 
\[
N(K) \leq \exp\bigg( - \frac{\Omega(d)}{L_K^2} \bigg) \cdot 4^d
\]
as $d \to \infty$.
\end{theorem}

One of the key innovations of~\cite{HSTV} was a method of deducing bounds on the covering number $N(K)$ from  bounds on the Kövner--Besicovitch measure of symmetry 
$$\Delta_{KB}(K) := \max_{x \in \R^d} \frac{|K\cap (x-K)|}{|K|}\, .$$
In particular, the authors of~\cite{HSTV} improved the (straightforward, but until then best known) lower bound $\Delta_{KB}(K) \ge 2^{-d}$ by a factor of $e^{\Omega(\sqrt{d})}$, and used that bound to prove~\eqref{eq:HSTV}. We will similarly deduce Theorem~\ref{main_thm} from the following lower bound on $\Delta_{KB}(K)$. 

\begin{theorem}\label{thm:DeltaKB}
If $K \subset \R^d$ is a convex body, then 
\[
\Delta_{KB}(K) \, \geq \, \exp\bigg( \frac{d}{2^{15} L_K^2} \bigg) \cdot 2^{-d}.
\] 
\end{theorem}


In addition to the application to Hadwiger's conjecture described above, our method 
also has an application to the geometry of numbers. To be precise, Ehrhart~\cite{E64} conjectured in 1964 that a convex body in $\R^d$ centred at the origin\footnote{We say that a convex body $K \subset \R^d$ is \emph{centred} at its centre of mass $\Ex[X]$, where $X \sim \U(K)$.} 
whose interior contains no lattice point other than the origin has volume at most $(d+1)^d / d!$ (this bound is attained by a simplex). The best-known upper bound for the volume of $K$ is of the form $e^{-\Omega(\sqrt{d})} \cdot 4^d$, obtained by Huang, Slomka, Tkocz and Vritsiou~\cite{HSTV}. We will use the bound on $L_K$ proved by Klartag and Lehec~\cite{KL} to deduce the following almost-exponential improvement of their bound. 

\begin{theorem}\label{cor:ehrhart}
Let $K \subset \R^d$ be a convex body centred at the origin. If $K \cap \Z^d = \{0\}$, then 
\[
|K| \leq \exp\bigg( - \Omega\bigg( \frac{d}{(\log d)^8} \bigg) \bigg) \cdot 4^d
\]
as $d \to \infty$.
\end{theorem}

In order to prove Theorem~\ref{cor:ehrhart}, we will need a variant of Theorem~\ref{thm:DeltaKB} that provides a similar lower bound on the ratio $|K \cap (-K)| / |K|$ (see Theorem~\ref{thm:centred}). The application of such bounds to Ehrhart’s conjecture was first observed by Henk, Henze and Hernández Cifre~\cite{HHH}, who used the bound $|K \cap (-K)| / |K| \ge 2^{-d}$, due to Milman and Pajor~\cite{MP00}, together with Minkowski’s theorem, to prove an upper bound of $4^d$ for Ehrhart’s conjecture. 

The rest of this note is organised as follows. In Section~\ref{sec:KB} we will prove Theorem~\ref{thm:DeltaKB}, in Section~\ref{sec:Had} we will deduce Theorems~\ref{cor:hadwiger} and~\ref{main_thm}, and in Section~\ref{sec:Ehr} we will prove Theorem~\ref{cor:ehrhart}.

\section{Bounding the Kövner--Besicovitch measure}\label{sec:KB}

{\setstretch{1.2}

One of the key ideas introduced in~\cite{HSTV} was that a lower bound on $\Delta_{KB}(K)$ can be obtained by considering the maximum density of the random variable $X+Y$, where $X$ and $Y$ are independent uniform elements of $K$. More precisely, they made the following observation. We write $f_X$ for the probability density function of a random variable $X$. 

\begin{lemma}\label{lem:XY:obs}
Let $K \subset \R^d$ be a convex body of volume $1$, and let $X$ and $Y$ be independent uniformly-chosen random elements of $K$. Then, for any $z \in K$, 
$$f_{\frac{X+Y}{2}}(z) = 2^d \cdot \big| K \cap \big( 2z - K \big) \big|.$$
\end{lemma}

\begin{proof}
Observe first that
$$2^{-d} \cdot f_{\frac{X + Y}{2}}(z) \, = \, f_{X+Y}(2z)\, = \, \int_{x \in \mathbb{R}^d}f_{X}(x)f_{Y}(2z-x) \, dx.$$
Now simply note that
\begin{align*}
\int_{x\in \mathbb{R}^d} f_{X}(x)f_{Y}(2z-x) \, dx & \, = \, \int_{x\in \mathbb{R}^d}\1 \big[ x\in K \big] \1\big[ 2z - x \in K \big] \, dx\\
& \, = \, \int_{x\in \mathbb{R}^d}\1\big[ x\in K \cap (2z-K) \big] \, dx\, = \, |K\cap (2z-K)|,
\end{align*}
as claimed.
\end{proof}

It follows immediately from Lemma~\ref{lem:XY:obs} that if $|K| = 1$, then
\begin{equation}\label{eq:basic:bound}
\Delta_{KB}(K) \, \geq \, 2^{-d} \cdot \big\| f_{\frac{X+Y}{2}} \big\|_\infty \, \geq \, 2^{-d} \cdot \frac{\Pr\big(\frac{X+Y}{2} \in A \big)}{\Pr(X \in A)} 
\end{equation}
for any measurable set $A \subset \R^d$. In order to prove their lower bound on $\Delta_{KB}(K)$, the authors of~\cite{HSTV} observed that the random variable $\big\| \frac{X+Y}{2} \big\|_2$ is typically about $\sqrt{2}$ times smaller than $\| X \|_2$, and applied the inequality~\eqref{eq:basic:bound} to a ball $A$ with radius halfway between these two typical values. They then used a `thin-shell' theorem of Guédon and Milman~\cite{GM}, which implies that if $K \subset \R^d$ is a convex body in isotropic position 
then, for any fixed $c > 0$, 
\begin{equation}\label{thm:GM}
\Pr\Big( \big| \|X\|_2 - \sqrt{d} \big| \ge c \sqrt{d} \Big) \le \exp\Big( - \Omega\big( \sqrt{d} \big) \Big),
\end{equation}
to deduce that $\Pr\big(\frac{X+Y}{2} \in A \big) \approx 1$ and $\Pr(X \in A) \le e^{-\Omega(\sqrt{d})}$ for this set $A$, giving their bound 
$$\Delta_{KB}(K) \geq e^{\Omega(\sqrt{d})} \cdot 2^{-d}.$$

The 
Guédon--Milman bound~\eqref{thm:GM} is best possible 
(to see this, consider the simplex), so it may seem at first sight that there is not much hope of using the method of~\cite{HSTV} to prove a significantly stronger lower bound on $\Delta_{KB}(K)$. In order to do so, we will replace the thin-shell estimate~\eqref{thm:GM} by a `small-ball' bound which depends on $L_K$, and the random variable $X + Y$ by a sum of arbitrarily many independent random variables. 
}
{\setstretch{1.13}

To be more precise, let $X_1,X_2,\ldots$ be a sequence of independent random variables, each chosen uniformly at random from the set $K$, and for each $k \in \N$, define 
\begin{equation}\label{def:Sk}
S_k := \frac{1}{2^k} \sum_{i = 1}^{2^k} X_i.
\end{equation}
Since $K$ is convex, it follows from the Prékopa--Leindler inequality that $f_{S_k}$ is log-concave. 
 
The key step is the following lemma, which bounds $f_{S_k}(z)$ in terms of $f_{\frac{X+Y}{2}}(z)$.

\begin{lemma}\label{iteratedsumcontrol}
For any convex body $K \subset \R^d$ with volume $1$, we have
$$f_{S_k}(z) \, \leq \, \Big( f_{\frac{X+Y}{2}}(z) \Big)^{2^k-1}$$
for all $z \in \R^d$ and every $k \in \N$.
\end{lemma}

\begin{proof}
The proof is by induction on $k$. Note that the conclusion holds trivially in the case $k = 1$, so let $k \ge 1$ and assume that the inequality holds for $k$; we will prove that it holds for $k+1$. Define $T_k := 2^{-k} \sum_{i = 1}^{2^k} X_{2^k + i}$, and note that $S_{k+1} = \frac{S_k + T_k}{2}$, and that $S_k$ and $T_k$ are independent and identically distributed random variables with support $K$. It follows that
$$f_{S_{k+1}}(z) = f_{\frac{S_k+T_k}{2}}(z) = 2^d \int_{y \in K} f_{S_k}(y) f_{S_k}(2z - y) \, dy$$
for every $z \in K$. Moreover, since $f_X$ and $f_Y$ are indicator functions on $K$, and $f_{S_k}$ is a log-concave function supported on $K$, 
$$\int_{y\in K} f_{S_k}(y) f_{S_k}(2z - y) \, dy \, \leq \, f_{S_k}(z)^2 \int_{y \in K}  f_X(y)  f_Y(2z - y) \, dy.$$
Now, by the induction hypothesis, we have 
$$f_{S_k}(z) \, \leq \, \Big( f_{\frac{X+Y}{2}}(z) \Big)^{2^k - 1},$$
and therefore, noting again that
$$\int_{y \in K}  f_X(y)  f_Y(2z - y) \, dy \,=\, 2^{-d} \cdot f_{\frac{X + Y}{2}}(z),$$
we obtain
$$f_{S_{k+1}}(z) \, \leq \, \Big( f_{\frac{X+Y}{2}}(z) \Big)^{2(2^k - 1)} f_{\frac{X + Y}{2}}(z) \, = \, \Big( f_{\frac{X+Y}{2}}(z) \Big)^{2^{k+1}-1},$$
for every $z \in K$, as required.
\end{proof}

We remark that in order to prove~\Cref{thm:DeltaKB} (and hence also Theorems~\ref{cor:hadwiger} and~\ref{main_thm}) we will only need the inequality
$$\| f_{S_k} \|_\infty \, \leq \, \Big( \big\| f_{\frac{X+Y}{2}} \big\|_\infty \Big)^{2^k-1}.$$
However, in the proof of Theorem~\ref{cor:ehrhart} we shall require the full strength of Lemma~\ref{iteratedsumcontrol}. 

Recall that, for any convex body $K \subset \R^d$, there exists an affine transformation that maps $K$ to a convex body $K'$ of volume $1$ such that $\Sigma_{K'} = L_K^2 I_d$, where $\Sigma_{K'} =\Ex[ X \otimes X ]$ is the covariance matrix of the uniform random variable $X \sim \cU(K')$, and $I_d$ is the identity matrix. For such a convex body $K'$, it is straightforward to calculate the covariance matrix of $S_k$.
  
}
  
\begin{lemma}\label{sumconcentration}
Let $K$ be a convex body, let $X \sim \cU(K)$, and suppose that $\Ex[ X \otimes X ] = L_K^2 I_d$. Then
$$\Ex\big[ S_k \otimes S_k \big] = 2^{-k} L_{K}^2 I_d$$
for every $k \in \N$.
\end{lemma}

\begin{proof}
Since $S_k = 2^{-k} \sum_{i = 1}^{2^k} X_i$ and the $X_i$ are uniform and independent, it follows that
$$\Ex\big[ S_k \otimes S_k \big]  = \frac{1}{2^{2k}}\sum_{i,j = 1}^{2^k} \Ex\big[ X_i \otimes X_j \big] \, = \, \frac{1}{2^{2k}} \sum_{i = 1}^{2^k} \Ex\big[ X_i \otimes X_i \big] = 2^{-k} L_{K}^2 I_d,$$
as claimed.
\end{proof}

We are now ready to prove~\Cref{thm:DeltaKB}.

\begin{proof}[Proof of~\Cref{thm:DeltaKB}]
By applying an affine transformation, we may assume that $K$ has volume $1$ and is centred at the origin, and that $\Ex[ X \otimes X ] = L_K^2 I_d$, where $X \sim \cU(K)$. Fix $k \in \N$ such that
$$2^{15} L_{K}^2 \le 2^{k} \le 2^{16} L_{K}^2,$$ 
set $R := 2^{-7} \sqrt{d}$, and observe that, by Markov's inequality and~\Cref{sumconcentration}, we have 
\[
\Pr\big( \| S_k \|_2 \geq R \big) \, \leq \, \frac{2^{14}}{d} \cdot \Ex\big[ \| S_k \|_2^2 \big] \, = \, \frac{2^{14}}{d} \cdot \sum_{i = 1}^{d} 2^{-k} L_{K}^2 \, = \, \frac{2^{14} L_K^2}{2^k} \, \leq \, \frac{1}{2}.
\]
Moreover, bounding $\Pr\big( \| X \|_2 \le R \big)$ simply by the volume of the ball of radius $R$, we obtain
$$\Pr\big( \| X \|_2 \le R \big) \, \le \, \frac{\pi^{d/2} R^d}{\Gamma(\frac{d}{2}+1)} \, \le \, \bigg( \frac{2e\pi R^2}{d} \bigg)^{d/2} \, \le \, e^{-2d-1}.$$
Combining these two bounds, we deduce that 
\begin{equation}\label{eq:fSk:lowerbound}
\| f_{S_k} \|_\infty \, \ge \, \frac{\Pr\big( \| S_k \|_2 \le R \big)}{\Pr\big( \| X \|_2 \le R \big)} \, \ge \, \frac{e^{2d+1}}{2} \, \ge \, e^{2d}.
\end{equation}
Now, by Lemma~\ref{iteratedsumcontrol}, it follows that 
$$\big\| f_{\frac{X+Y}{2}} \big\|_\infty \, \geq \, \big( \| f_{S_k} \|_\infty \big)^{1/(2^{k}-1)} \, \geq \, e^{d / 2^{k-1}},$$
and hence, by Lemma~\ref{lem:XY:obs} and since $2^{k} \leq 2^{16} L_{K}^2$, we obtain
$$\Delta_{KB}(K) \, \geq \, 2^{-d} \cdot \big\| f_{\frac{X+Y}{2}} \big\|_\infty \, \geq \, \exp\bigg( \frac{d}{2^{15} L_K^2} \bigg) \cdot 2^{-d},$$
as required.
\end{proof}

We remark that the constant $2^{-15}$ in~\Cref{thm:DeltaKB} could be improved somewhat by taking $R$ a little larger (and thus $k$ a little smaller); however, we shall need~\eqref{eq:fSk:lowerbound} again in Section~\ref{sec:Ehr}, and we chose the constants in the proof above with the application there in mind. 


\section{Hadwiger's conjecture}\label{sec:Had}

In this section we will deduce Theorems~\ref{cor:hadwiger} and~\ref{main_thm} from~\Cref{thm:DeltaKB}. We~begin with the proof of~\Cref{main_thm}, for which we will need the following asymmetric variant of $N(K)$: given convex bodies $A$ and $B$ in $\R^d$, define
$$ N(A,B) = \min\Big\{ N \in \N \,:\, \exists\, x_1,\ldots, x_N \in \R^d \,\text{ such that }\, A \subset \, \bigcup_{i=1}^N \big( x_i + \inte(B) \big) \Big\}.$$
We will use the following classical fact (see~\cite{RZ} or~\cite[Corollary~3.5]{N18}), which follows from Rogers' bound~\cite{R57} on the density of coverings of $\R^d$ with translates of convex bodies. 

\begin{lemma}\label{lem_main}
If $A,B \subset \R^d$ are convex bodies, then
$$N(A,B) \le O\big( d \log d \big) \cdot \frac{|A-B|}{|B|}.$$
\end{lemma}

We are now ready to deduce Theorem~\ref{main_thm} from Theorem~\ref{thm:DeltaKB}.

\begin{proof}[Proof of~\Cref{main_thm}]
By~\Cref{thm:DeltaKB}, there exists $x \in \R^d$ such that 
\begin{equation}\label{eq:KB:app}
\frac{|K \cap (x - K)|}{|K|} \, \geq \, \exp\bigg( \frac{d}{2^{15} L_K^2} \bigg) \cdot 2^{-d}.
\end{equation}
Set $S := K \cap (x - K)$, and note that
$$N(K) \leq N(K,S) \qquad \text{and} \qquad |K - S| \le |K + K| = 2^d \cdot |K|,$$
since $S \subset K$ and $S \subset x - K$, respectively. It therefore follows from~\Cref{lem_main} that
$$N(K) \leq N(K,S) \le O( d \log d ) \cdot \frac{|K-S|}{|S|} \, \le \, O( d \log d ) \cdot 2^d \cdot \frac{|K|}{|S|},$$
and hence, by~\eqref{eq:KB:app}, we obtain
$$N(K) \le O( d \log d ) \cdot \exp\bigg( - \frac{d}{2^{15} L_K^2} \bigg) \cdot 4^d \, = \, \exp\bigg( - \frac{\Omega(d)}{L_K^2} \bigg) \cdot 4^d$$
as $d \to \infty$, as required.
\end{proof}

In order to deduce Theorem~\ref{cor:hadwiger} and Theorem~\ref{cor:ehrhart}, we will need the following theorem of Klartag and Lehec~\cite{KL}.


\begin{theorem}
\label{thm:Bourgain:polylog}
If $K \subset \R^d$ is a convex body, then 
\[
L_K \, = \, O(\log d)^4.
\] 
\end{theorem}

Theorem~\ref{cor:hadwiger} now follows immediately. 

\begin{proof}[Proof of Theorem~\ref{cor:hadwiger}]
By Theorems~\ref{main_thm} and~\ref{thm:Bourgain:polylog}, it follows that
$$N(K) \, \leq \, \exp\bigg( - \frac{\Omega(d)}{L_K^2} \bigg) \cdot 4^d \, \le \, \exp\bigg( - \Omega\bigg( \frac{d}{(\log d)^8} \bigg) \bigg) \cdot 4^d$$
as $d \to \infty$, as required.
\end{proof}

\section{Ehrhart’s conjecture}\label{sec:Ehr}

In order to prove Theorem~\ref{cor:ehrhart}, we will need the following variant of Theorem~\ref{thm:DeltaKB}.

\begin{theorem}\label{thm:centred}
If $K \subset \R^d$ is a convex body centred at the origin, then 
\[
\Delta_{KB}(K) \, \geq \, \frac{| K \cap (-K)|}{|K|} \, \geq \, \exp\bigg( \frac{d}{2^{16} L_K^2} \bigg) \cdot 2^{-d}.
\] 
\end{theorem}

We will deduce Theorem~\ref{thm:centred} from the proof of Theorem~\ref{thm:DeltaKB}, together with the following bound on the value of a log-concave function at its centre of mass~\cite[Theorem~4]{F97}.

\begin{theorem}
\label{thm:centre}
If $f \colon \R^d \to \R_+$ is a log-concave function, then
$$f(y) \ge e^{-d} \cdot \| f \|_\infty,$$
where $y = \int_{x \in \R^d} f(x) \cdot x \, dx$ is the centre of mass of $f$. 
\end{theorem}



Theorem~\ref{thm:centred} now follows from Lemma~\ref{iteratedsumcontrol}, as before.

\begin{proof}[Proof of~\Cref{thm:centred}]
Recall that the function $f_{S_k}$ is log-concave, where $S_k$ is the random variable defined in~\eqref{def:Sk}, and note that, since $K$ is centred at the origin, the centre of mass of $f_{S_k}$ is also the origin. By Theorem~\ref{thm:centre} and~\eqref{eq:fSk:lowerbound}, it follows that
$$f_{S_k}(0) \ge e^{-d} \cdot \| f_{S_k} \|_\infty \, \ge \, e^d.$$
Now, by Lemma~\ref{iteratedsumcontrol}, it follows that 
$$f_{\frac{X+Y}{2}}(0) \, \ge \, \big( f_{S_k}(0) \big)^{1/(2^{k}-1)} \, \geq \, e^{d / 2^k},$$
and hence, by Lemma~\ref{lem:XY:obs} and since $2^{k} \leq 2^{16} L_{K}^2$, we obtain
$$\frac{|K \cap (-K)|}{|K|} \, = \, 2^{-d} \cdot f_{\frac{X+Y}{2}}(0) \, \geq \, \exp\bigg( \frac{d}{2^{16} L_K^2} \bigg) \cdot 2^{-d},$$
as claimed.
\end{proof}

Finally, to deduce Theorem~\ref{cor:ehrhart}, recall that, by Minkowski's theorem, every convex body $K \subset \R^d$ such that $K = -K$ and $K \cap \Z^d = \{0\}$ has volume at most $2^d$. 

\begin{proof}[Proof of Theorem~\ref{cor:ehrhart}]
By Minkowski's inequality and Theorems~\ref{thm:Bourgain:polylog} and~\ref{thm:centred}, we have 
$$\frac{2^d}{|K|} \, \ge \, \frac{|K \cap (-K)|}{|K|} \, \ge \, \exp\bigg( \frac{d}{2^{16} L_K^2} \bigg) \cdot 2^{-d} \, \ge \, \exp\bigg( \Omega\bigg( \frac{d}{(\log d)^8} \bigg) \bigg) \cdot 2^{-d}$$
as $d \to \infty$, as required.
\end{proof}

\section*{Acknowledgements}

This research began while the second author was visiting Rio de Janeiro, and the authors are grateful to IMPA for providing an excellent working environment. They would also like to thank Grigoris Paouris for pointing them to reference~\cite{F97}.

{\setstretch{1.25}

}

\end{document}